\documentclass[12pt]{article}

\usepackage[demo]{graphicx} 
\usepackage{hyperref}
\usepackage{float}
\usepackage{indentfirst}
\usepackage{tikz}
\usepackage{amssymb}
\usepackage{amsmath,mathrsfs,amsfonts}
\usepackage{graphics}
\usepackage{amsthm}
\usepackage{setspace}
\usepackage{CJK,cases}
\usepackage{indentfirst}
\usepackage{setspace}
\usepackage[mathscr]{eucal}
\usepackage{caption}
\usepackage{subcaption}
\usepackage{rotating}
\usetikzlibrary{decorations.pathreplacing}
\newtheorem{theo}{Theorem}
\newtheorem{lem}[theo]{Lemma}
\newtheorem{corollary}[theo]{Corollary}

\newtheorem{problem}[theo]{Problem}
 \theoremstyle{definition}

\theoremstyle{remark}

\newcounter{casenum}[theo]

\newcounter{subcasenum}[theo]

\newcounter{claimnum}[theo]

\allowdisplaybreaks
\begin{document}
\thispagestyle{plain}

\begin{center} {\Large Connected graphs with a given dissociation number attaining the minimum spectral radius
}
\end{center}
\pagestyle{plain}
\begin{center}
{
  {\small  Zejun Huang$^{1}$, Jiahui Liu$^{1}$\footnote{Corresponding author. \\ \indent~~ Email: mathzejun@gmail.com (Huang), mathjiahui@163.com (Liu), 2015110112@email.szu.edu.cn (Zhang)}, Xinwei Zhang$^{1}$}\\[3mm]
  {\small 1. School of Mathematical Sciences, Shenzhen University, Shenzhen 518060, China }\\

}
\end{center}

\begin{center}

\begin{minipage}{140mm}
\begin{center}
{\bf Abstract}
\end{center}
{\small    A dissociation set of a graph  is a set of vertices which induces a subgraph with maximum degree less than or equal to one. The dissociation number of a graph  is the maximum cardinality of its dissociation sets. In this paper, we study the   connected graphs of order $n$ with a given dissociation number that attains the minimum spectral radius. We characterize these graphs when the dissociation number is in $\{n-1,~n-2,~\lceil2n/3\rceil,~\lfloor2n/3\rfloor,~2\}$. We also prove that these graphs are trees when the dissociation number is larger than $\lceil {2n}/{3}\rceil$.

 }

\end{minipage}
\end{center}

\section{Introduction and main results}

   Graphs in this paper are simple.  For a graph $G$, we denote by $V(G)$ its  vertex set and $E(G)$ its edge set. The cardinalities of $V(G)$ and $E(G)$ are called the {\it order} and the {\it size} of $G$, respectively. For a set $S\subseteq V(G)$,  $G[S]$ is  the subgraph of $G$ induced by $S$. We denote by $G-S$ the induced subgraph $G[V\setminus S]$, which is also denoted by $G-v$ if $S=\{v\}$ is a singleton.   If $e\in E(G)$, $G-e$ is the graph obtained from $G$ by deleting the edge $e$; if $e\not\in E(G)$, $G+e$ is the graph obtained from $G$ by adding the edge $e$.

The {\it union} of two graphs $G_1$ and $G_2$ is the graph $G_1\cup G_2$ with vertex set $V(G_1)\cup V(G_2)$ and edge set $E(G_1)\cup E(G_2)$, which is also called  a {\it disjoint union} and denoted by $G_1+G_2$ if $V(G_1)\cap V(G_2)=\emptyset$. We denote by $kG$ the  disjoint union of $k$ copies of $G$.
The {\it join} of    two vertex disjoint graphs $G_1$ and $G_2$, denoted by $G_1\vee G_2$, is obtained from $G_1+G_2$ by adding all edges $uv$ with $u\in V(G_1)$ and $v\in V(G_2)$.

Given two graphs $G$ and $H$, if there  is a  bijection  $\phi: V(G)\rightarrow V(H)$  such that $uv\in E(G)$ if and only if $\phi(u)\phi(v)\in E(H)$, then $G$ and $H$ are  said to be {\it isomorphic}, written $G\cong H$.

 The {\it spectral radius} of a graph $G$, denoted by $\rho(G)$, is the largest eigenvalue of its adjacency matrix $A(G)$. When $G$ is connected, by the Perron-Frobenius theorem (see \cite{Zhan}), there is a unique positive unit eigenvector of $A(G)$ corresponding to the eigenvalue $\rho(G)$, which is called the {\it Perron vector} of $G$.

Given a set of graphs, it is natural to consider the possible values of certain parameters on these graphs. In 1986, Brualdi and Solheid \cite{BS} proposed the general problem on determining the maximum or minimum spectral radius   of graphs from a given set as well as the extremal graphs attaining the maximum or minimum spectral radius. A lot  of works have been done along this line. Brualdi and Hoffman \cite{BRUALDI1985133} determined the maximum spectral radius of a graph $G$ of order $n$ with size exactly
 ${s\choose 2}$, which is attained if and only if $G\cong K_s+ (n-s)K_1$.
Hong, Shu and Fang \cite{HONG2001177}
 determined the maximum spectral radius of connected graphs with given order, size and minimum degree as well as the extremal graphs. Berman and Zhang \cite{BERMAN2001233} characterized  the graphs  with the maximum spectral radius among connected graphs of order  $n$ with a given number of cut vertices.  Feng, Yu and Zhang \cite{FENG2007133} characterized  the graphs with the maximum spectral radius among  graphs  of order $n$ with a given  matching number. Liu, Lu and Tian \cite{LIU2007449} determined the graphs with the maximum spectral radius among unicyclic graphs of order
$n$ with a given  diameter. Van Dam and Kooij \cite{VANDAM2007408} characterized the connected graphs with the minimum spectral radius among    connected graphs of order $n$ with a given diameter from $ \{1,~2,~\lfloor\frac{n}{2}\rfloor,~n-3,~n-2,~n-1\}$.

On graphs with a given independent number, the characterization of such connected graphs of order $n$ attaining the maximum spectral radius is trivial. However, it is not easy to characterize the connected graphs of order $n$ with a given independence number $\alpha$ that attains the minimum spectral radius.   Xu, Hong, Shu and Zhai \cite{XU2009937} proposed this problem in 2009 and they solved the problem for the cases $\alpha\in\{1,~2,~\lceil\frac{n}{2}\rceil,~\lceil\frac{n}{2}\rceil+1,~n-3,~n-2,~n-1\}$; Du and Shi \cite{Du2013GraphsWS} solved the problem when  $\alpha=3,4$ and the order $n$ is divided by $\alpha$;   Lou and Guo \cite{LOU2022112778} solved the problem for the case  $\alpha=n-4$ and they proved that the extremal graphs are trees when $\alpha\ge \lceil n/2\rceil$;   Hu, Huang and Lou \cite{hu2022graphs} solved the problem for the cases $\alpha\in\{n-5,~n-6\}$, and they gave the structural features for the  extremal graphs in detail as well as    a constructing theorem for the extremal graphs when $\alpha\ge\lceil\frac{n}{2}\rceil$; Choi and Park \cite{choi2023minimal}  solved the problem for the case $\alpha=\lceil\frac{n}{2}\rceil-1$. For the other cases, the problem is still open.

A set $S\subseteq V(G)$ is called  a {\it dissociation set} if the induced subgraph $G[S]$ has maximum degree at most $1$. A {\it maximum dissociation set} is a  dissociation set  of $G$ with maximum cardinality. The {\it dissociation  number} of a graph $G$, denoted by $diss(G)$, is the cardinality of   a maximum dissociation set in $G$. The dissociation set is closely related to the {\it vertex $k$-path cover}, which is a set of vertices intersecting every path of order $k$ in $G$. Notice that the vertex $k$-path cover   has applications in   secure communication in  wireless sensor networks; see \cite{BKKS,CHZ}. Problems on the dissociation number and dissociation sets have been widely studied; see \cite{bock2022bound,bock2022relating,das2023spectral,ORLOVICH20111352,SL,TU2022127107, TZS, doi:10.1137/0210022}.

The dissociation number is a natural generalization of the independence number.
Similarly to the minimum spectral radius problem on the independence number, we study the following problem in this paper.
\begin{problem}
Let $n,k$ be integers such that $2\le k\le n$.
Denote by $ \mathcal{G}_{n,k}$ the connected graphs of order $n$ with dissociation number $k$. Characterize the graphs in  $ \mathcal{G}_{n,k}$ that attain the minimum spectral radius.
\end{problem}

Notice that the characterization of graphs in  $ \mathcal{G}_{n,k}$ attaining the maximum spectral radius is trivial. In fact, suppose $G\in \mathcal{G}_{n,k}$. Then by choosing an arbitrary   dissociation set  with cardinality $k$, it is not difficult to prove that $G$ attains the maximum spectral radius if and only if $G\cong K_{n-k}\vee  \frac{k}{2}K_2$ when $k$ is even and $G\cong K_{n-k}\vee  (\frac{k-1}{2}K_2+K_1)$ when $k$ is odd.

   We denote by $P_n$ and $C_n$ the path and the cycle of order $n$, respectively. The path $P_{n+1}$ is also called an  {\it $n$-path}, since its length is $n$. In a tree $T$, a vertex with degree larger than 2 is called a {\it branch vertex}. If $v$ is a branch vertex of a tree $T$ such that $T-v$ has at most one component containing a branch vertex, then $v$ is called an {\it end branch vertex}.
   If a $k$-path $P$ is an induced subgraph of $T$ with one of its end adjacent to a branch vertex $v\in V(T)$, then $P$ is called a {\it branch  $k$-path} attached to $v$, which is also called a  {\it branch edge } attached to $v$ if $P$ is an edge. If a leaf is adjacent to a branch vertex   $v\in V(T)$, we say the leaf is   attached  to $v$.

   A {\it complete $k$-partite graph} $G$ of order $n$ is a graph whose vertex set can be partitioned into $k$ subsets $V_1,\ldots,V_k$ such that $uv\in E(G)$ if and only if $u\in V_i,v\in V_j$ with $i\ne j$. Moreover, if  $|V_i|\in \{\lceil n/k\rceil,\lfloor n/k\rfloor\}$ for $i=1,2,\ldots,k$, $G$  is said to be {\it balanced}.

Denote by $S_n$ the star graph of order $n$. Let $S(r,t)$ be the graph obtained from $S_{r+1}$ by attaching $t$  branch edges to the center of $S_{r+1}$.
If $n$ is even, we denote by $H(n)$  the graph obtained from $P_4$ by attaching $\lceil(n-4)/2\rceil$ and $\lfloor(n-4)/2\rfloor$ branch edges to the two ends  of $P_4$  respectively; if $n$ is odd, we denote by $H(n)$ the graph obtained from $H(n-1)$ by attaching a   leaf to the branch vertex of $H(n-1)$ with degree $\lfloor (n-5)/2\rfloor+1$. What follows are the diagrams of $S(r,t)$ and $H(n)$.

\begin{figure}[H]
	\centering
	\begin{subfigure}{.2\textwidth}
		\centering
			\begin{tikzpicture}
			\draw (0,0)[fill]circle[radius=1.0mm]--(0,1)[fill]circle[radius=1.0mm];
			\draw (0,0)[fill]circle[radius=1.0mm]--(0,-1)[fill]circle[radius=1.0mm];
			\draw(0,-2)[fill]circle[radius=1.0mm]--(0,-1);
			\draw (0,0)[fill]circle[radius=1.0mm]--(0.8,0.7)[fill]circle[radius=1.0mm];
			\draw (0,0)[fill]circle[radius=1.0mm]--(0.7,-0.75)[fill]circle[radius=1.0mm];
			\draw (1.35,-1.4)[fill]circle[radius=1.0mm]--(0.7,-0.75)[fill]circle[radius=1.0mm];
			\node at (0.39,0.93){\begin{turn}{10}$\ddots$\end{turn}};
			\node at (0.4,-0.93){\begin{turn}{25}$\cdots$\end{turn}};
			\draw[decorate,decoration={brace,mirror,amplitude=0.2cm}](0.8,0.82)--(0,1.12);
			\draw[decorate,decoration={brace,mirror,amplitude=0.2cm}](0,-2.13)--(1.35,-1.53);
			\node at (1,-1.8)[below=3pt]{$t$};
			\node at (0.69,0.85)[above=3pt]{$r$};
			\end{tikzpicture}
		\caption*{$S(r,t)$}
		\label{fig:sub7.1}
	\end{subfigure}%
\begin{subfigure}{.33\textwidth}
	\centering
	\vspace{11pt}
	\begin{tikzpicture}
	\foreach \y in {2}{
		\foreach \x in {3,4,6,7}{
			\draw (\x, \y)[fill]circle[radius=1.0mm];
			
		}
	};
	\draw(4,3.5)[fill]circle[radius=1.0mm]--(6,3.5)[fill]circle[radius=1.0mm];
	
	\foreach \y in {1}{
		\foreach \x in {3,4,7,6}{
			\draw (\x, \y)[fill]circle[radius=1.0mm];
			
		}
	};
	\foreach \y in{2}{
		\foreach \x in {3,...,4}{
			\draw(\x,\y)--(4,3.5);
		}
	};
	\foreach \y in{2}{
		\foreach \x in {6,...,7}{
			\draw(\x,\y)--(6,3.5);
		}
	};
	\foreach \y in{1}{
		\foreach \x in {3,4,7,6}{
			\draw(\x,\y)--(\x,2);}};
	\draw[decorate,decoration={brace,mirror,amplitude=0.2cm}](3,0.85)--(4,0.85);
	\node at(3.5,0.8)[below=3pt]{$\lceil\frac{n-4}{2}\rceil$};
	\draw[decorate,decoration={brace,mirror,amplitude=0.2cm}](6,0.85)--(7,0.85);
	\node at(6.5,0.8)[below=3pt]{$\lfloor\frac{n-4}{2}\rfloor$};
	\node at (3.5,1){$\cdots$};
	\node at ((6.5,1){$\cdots$};;
	\draw(4.7,3.5)[fill]circle[radius=1.0mm];
	\draw(5.35,3.5)[fill]circle[radius=1.0mm];
	\end{tikzpicture}
	\caption*{$H(n)$, $n$ is even}
	\label{fig:111}
\end{subfigure}
	\begin{subfigure}{.33\textwidth}
	\centering
	\vspace{11pt}
	\begin{tikzpicture}
	\foreach \y in {2}{
		\foreach \x in {3,4,6,7}{
			\draw (\x, \y)[fill]circle[radius=1.0mm];
			
		}
	};
\draw (7,3.5)[fill]circle[radius=1.0mm]--(6,3.5);
	\draw(4,3.5)[fill]circle[radius=1.0mm]--(6,3.5)[fill]circle[radius=1.0mm];
	
	\foreach \y in {1}{
		\foreach \x in {3,4,7,6}{
			\draw (\x, \y)[fill]circle[radius=1.0mm];
			
		}
	};
	\foreach \y in{2}{
		\foreach \x in {3,...,4}{
			\draw(\x,\y)--(4,3.5);
		}
	};
	\foreach \y in{2}{
		\foreach \x in {6,...,7}{
			\draw(\x,\y)--(6,3.5);
		}
	};
	\foreach \y in{1}{
		\foreach \x in {3,4,7,6}{
			\draw(\x,\y)--(\x,2);}};
	\draw[decorate,decoration={brace,mirror,amplitude=0.2cm}](3,0.85)--(4,0.85);
	\node at(3.5,0.8)[below=3pt]{$\lceil\frac{n-5}{2}\rceil$};
	\draw[decorate,decoration={brace,mirror,amplitude=0.2cm}](6,0.85)--(7,0.85);
	\node at(6.5,0.8)[below=3pt]{$\lfloor\frac{n-5}{2}\rfloor$};
	\node at (3.5,1){$\cdots$};
	\node at ((6.5,1){$\cdots$};;
	\draw(4.7,3.5)[fill]circle[radius=1.0mm];
	\draw(5.35,3.5)[fill]circle[radius=1.0mm];
	\end{tikzpicture}
	\caption*{$H(n)$, $n$ is odd}
	\label{fig:123}
\end{subfigure}
	\caption{The graphs $S(r,t)$ and $H(n)$}
\label{fig:7}
\end{figure}
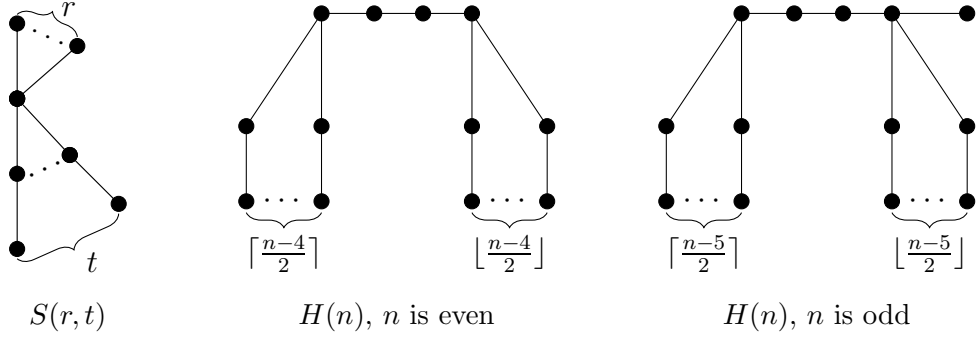

Our main results state  as follows.

\begin{theo}\label{th2}
Let $n$ and $k$ be positive integers with $ k>\lceil\frac{2n}{3}\rceil $.	If $G$ attains the minimum spectral radius  in $\mathcal{G}_{n,k}$, then $G$ is a tree.
\end{theo}

\begin{theo}\label{th1}
Let $n$ and $k$ be positive integers with $2\le k\le n-1$. Suppose $G$ attains the minimum spectral radius in $\mathcal{G}_{n,k}$.

(i) If $k=n-1$, then $G\cong   S(r,\lfloor(n-1)/2\rfloor),$ where $r=0$ if $n$ is odd and $r=1$ if $n$ is even.

(ii) If $k=n-2$ with $n\ge 10$, then $G\cong H(n)$.

(iii) If $k=\lceil 2n/3\rceil$, then $G$ is a path.

(iv) If $k=\lfloor 2n/3\rfloor$ with $n\ne 0 ~(mod ~3)$, then $G$ is  a cycle.

(v) If $k=2$, then $G$ is a balanced complete $\lceil n/2\rceil$-partite graph.
\end{theo}

\section{Preliminary}

In this section, we present some preliminary lemmas. Some of these lemmas concern  the dissociation number of graphs, which have independent interests.

\begin{lem}
	\label{lemma3}\cite{Hoffman_1975} If $G_2$ is a proper subgraph of a graph $G_1$, then $\rho(G_1)>\rho(G_2)$.		
\end{lem}
\begin{lem}
	\label{lemma2}\cite{0} Let $v$ be a vertex in a connected graph  $G$, and let $k, m$ be nonnegative integers such that $k\ge m$. Denote by $G_{k,m}$ the graph obtained from $G$ by attaching   two new paths   $vv_1v_2\dots v_k$ and  $vu_1u_2\dots u_m$  to $v$, where $v_1,\ldots, v_k,u_1,\ldots,u_m$ are distinct and $\{v_1,\ldots,v_k,u_1,\ldots,u_m\}\cap V(G)=\emptyset$.  Then  $\rho(G_{k,m})>\rho(G_{k+1,m-1})$.
\end{lem}

Let $v$ be a vertex in a graph $G$. We denote by $N_G(v)$ the set of neighbors of $v$ in $G$, which will also be abbreviated as $N(v)$.

\begin{lem}
\label{lemma13}\cite{WU2005343} Let $G$ be a connected graph with $V(G)=\{1,2,\ldots,n\}$. Suppose $u , v \in V(G)$  are distinct vertices  and $\{v_1,v_2,\ldots,v_s\} \subseteq N_G(v)\setminus N_G(u)$. Suppose $X=(x_1,x_2,\dots,x_n)^T$ is the Perron vector of $G$ with $x_k$   corresponding to the vertex $k$ for $1\le k\le n$. Denote by $G^*=G-\{vv_i|1\le i\le s\}+\{uv_i|1\le i\le s\}$. If $x_u\ge x_v$, then $\rho(G^*)>\rho(G)$. 	
\end{lem}

\begin{lem}
	\label{lemma1}\cite{Cvetkovi1995} Among all the connected graphs on $n$ vertices, the path $P_n$ has the minimum spectral radius; $\rho(P_n)=2cos\frac{\pi}{n+1}$.
\end{lem}
\begin{lem}
	\label{lemma6}\cite{Smith1970} The only connected graphs on $n$ vertices with spectral radius less than 2 are the path $P_n$ and
$T_n$, with additional cases $E_6, E_7, E_8$ when $n=6,7,8$, where $W_n, E_6, E_7, E_8$ have the following diagrams.
\end{lem}
\begin{figure}[H]
\begin{subfigure}{0.4\textwidth}
\centering
		\begin{tikzpicture}
		\draw(0,0)[fill]circle[radius=1.0mm]--(1,0)[fill]circle[radius=1.0mm];
		\draw(1,0)[fill]circle[radius=1.0mm]--(2,0)[fill]circle[radius=1.0mm];
		\draw(1,0)[fill]circle[radius=1.0mm]--(1,1)[fill]circle[radius=1.0mm];
		\node at (2.5,0){$\cdots$};
		\draw(3,0)[fill]circle[radius=1.0mm];
		\draw(3,0)--(4,0)[fill]circle[radius=1.0mm];
		\end{tikzpicture}
		\caption*{$W_n$ }
		\label{fig:1}
	\end{subfigure}
	\begin{subfigure}{0.6\textwidth}
		\centering
		\begin{tikzpicture}
		\draw(0,0)[fill]circle[radius=1.0mm]--(0.8,0)[fill]circle[radius=1.0mm];
		\draw(0.8,0)[fill]circle[radius=1.0mm]--(1.6,0)[fill]circle[radius=1.0mm];	
		\draw(1.6,0)[fill]circle[radius=1.0mm]--(2.4,0)[fill]circle[radius=1.0mm];
		\draw(2.4,0)[fill]circle[radius=1.0mm]--(3.2,0)[fill]circle[radius=1.0mm];
		\draw(1.6,0.8)[fill]circle[radius=1.0mm]--(1.6,0)[fill]circle[radius=1.0mm];
		\end{tikzpicture}
		\caption*{$E_6$ }
		\label{fig:sub2.1}
	\end{subfigure}
	\begin{subfigure}{0.4\textwidth}
		\centering
		\begin{tikzpicture}
		\draw(0,0)[fill]circle[radius=1.0mm]--(0.8,0)[fill]circle[radius=1.0mm];
		\draw(0.8,0)[fill]circle[radius=1.0mm]--(1.6,0)[fill]circle[radius=1.0mm];	
		\draw(1.6,0)[fill]circle[radius=1.0mm]--(2.4,0)[fill]circle[radius=1.0mm];
		\draw(2.4,0)[fill]circle[radius=1.0mm]--(3.2,0)[fill]circle[radius=1.0mm];
		\draw(3.2,0)[fill]circle[radius=1.0mm]--(4,0)[fill]circle[radius=1.0mm];
		\draw(1.6,0.8)[fill]circle[radius=1.0mm]--(1.6,0)[fill]circle[radius=1.0mm];
		\end{tikzpicture}
		\caption*{$E_7$}
		\label{fig:sub2.2}
	\end{subfigure}
	\begin{subfigure}{0.6\textwidth}
		\centering
		\begin{tikzpicture}
		\draw(0,0)[fill]circle[radius=1.0mm]--(0.8,0)[fill]circle[radius=1.0mm];
		\draw(0.8,0)[fill]circle[radius=1.0mm]--(1.6,0)[fill]circle[radius=1.0mm];	
		\draw(1.6,0)[fill]circle[radius=1.0mm]--(2.4,0)[fill]circle[radius=1.0mm];
		\draw(2.4,0)[fill]circle[radius=1.0mm]--(3.2,0)[fill]circle[radius=1.0mm];
		\draw(3.2,0)[fill]circle[radius=1.0mm]--(4,0)[fill]circle[radius=1.0mm];
		\draw(4,0)[fill]circle[radius=1.0mm]--(4.8,0)[fill]circle[radius=1.0mm];
		\draw(1.6,0.8)[fill]circle[radius=1.0mm]--(1.6,0)[fill]circle[radius=1.0mm];
		\end{tikzpicture}
		\caption*{$E_8$ }
		\label{fig:sub2.3}
	\end{subfigure}
\caption{The graphs $W_n$, $E_6,~E_7,~E_8$ }
	\label{fig:2}
\end{figure}
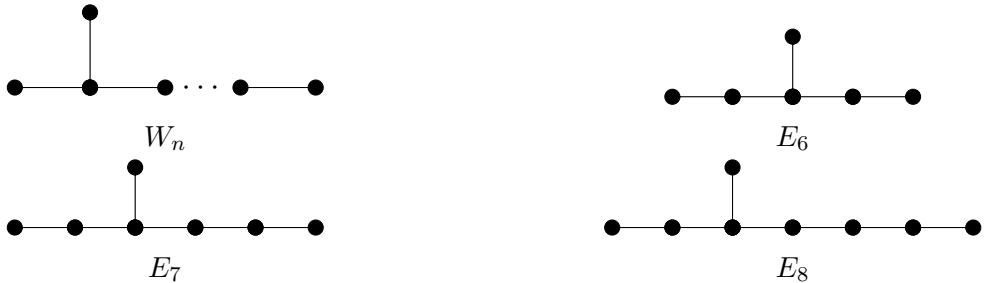

\begin{lem}
	\label{lemma7}\cite{Smith1970} The only connected graphs on $n$ vertices with spectral radius 2 are the cycle   $C_n$ and  $\tilde{W}_{n}$, with additional cases $\tilde{E}_6, \tilde{E}_7, \tilde{E}_8$ when $n=6,7,8$, where $\tilde{W}_{n}, \tilde{E}_6, \tilde{E}_7$ and $ \tilde{E}_8$ have the following diagrams.
\end{lem}
\begin{figure}[H]
	\centering
	\begin{subfigure}{0.45\textwidth}
		\vspace{15pt}
		\centering
		\begin{tikzpicture}
		\draw(0,0)[fill]circle[radius=1.0mm]--(1,0)[fill]circle[radius=1.0mm];
		\draw(1,0)[fill]circle[radius=1.0mm]--(2,0)[fill]circle[radius=1.0mm];
		\draw(3,0)[fill]circle[radius=1.0mm]--(4,0)[fill]circle[radius=1.0mm];
		\draw(4,0)[fill]circle[radius=1.0mm]--(5,0)[fill]circle[radius=1.0mm];
		\draw(1,1)[fill]circle[radius=1.0mm]--(1,0)[fill]circle[radius=1.0mm];
		\draw(4,1)[fill]circle[radius=1.0mm]--(4,0)[fill]circle[radius=1.0mm];
		\node at (2.5,0){$\cdots$};
		\end{tikzpicture}
		\caption*{$\tilde{W}_{n}$}
		\label{fig:3}
	\end{subfigure}
	\begin{subfigure}{0.45\textwidth}
		\centering
		\begin{tikzpicture}
		\foreach \y in {0}{
			\foreach \x in {0,0.8,1.6,2.4}{
				\draw (\x, \y)[fill]circle[radius=1.0mm]--(\x+0.8,\y);
				
			}
		};
		\draw(3.2,0)[fill]circle[radius=1.0mm];
		\draw(1.6,0.8)[fill]circle[radius=1.0mm]--(1.6,0);
		\draw(1.6,1.6)[fill]circle[radius=1.0mm]--(1.6,0.8);
		\end{tikzpicture}
		\caption*{$\tilde{E}_6$}
		\label{fig:4.1}
	\end{subfigure}
	\begin{subfigure}{0.45\textwidth}
		\centering
		\vspace{23pt}
		\begin{tikzpicture}
		\foreach \y in {0}{
			\foreach \x in {0,0.8,1.6,2.4,3.2,4}{
				\draw (\x, \y)[fill]circle[radius=1.0mm]--(\x+0.8,\y);
				
			}
		};
		\draw(4.8,0)[fill]circle[radius=1.0mm];
		\draw(2.4,0.8)[fill]circle[radius=1.0mm]--(2.4,0);
		\end{tikzpicture}
		\caption*{$\tilde{E}_7$}
		\label{fig:4.2}
	\end{subfigure}
	\begin{subfigure}{0.45\textwidth}
		\centering
		\vspace{23pt}
		\begin{tikzpicture}
		\foreach \y in {0}{
			\foreach \x in {0,0.8,1.6,2.4,3.2,4,4.8}{
				\draw (\x, \y)[fill]circle[radius=1.0mm]--(\x+0.8,\y);
				
			}
		};
		\draw(5.6,0)[fill]circle[radius=1.0mm];
		\draw(1.6,0.8)[fill]circle[radius=1.0mm]--(1.6,0);
		\end{tikzpicture}
		\caption*{$\tilde{E}_8$}
		\label{fig:4.3}
	\end{subfigure}
\label{fig:4}
	\caption{The graphs $\tilde{W}_{n}$, $\tilde{E}_6,\tilde{E}_7,\tilde{E}_8$}
\end{figure}
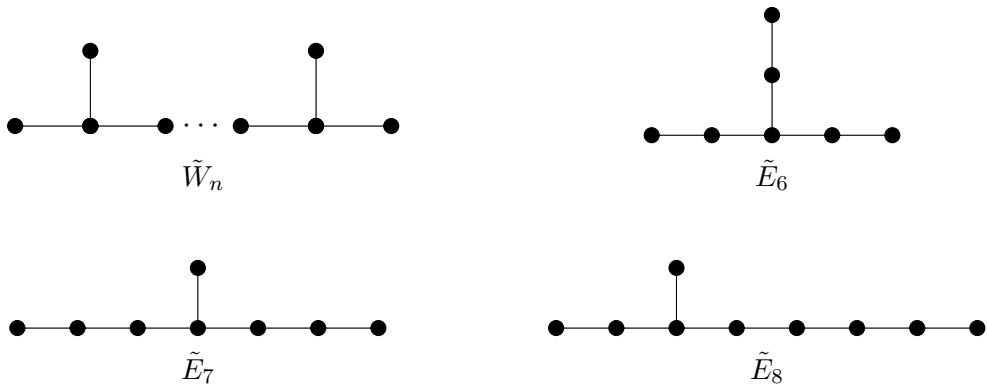
It is clear that $diss(E_6)=5$, $diss(E_7)=5$, $diss(E_8)=6$, $diss(\tilde{E}_6)=6$, $diss(\tilde{E}_7)=6$, $diss(\tilde{E}_8)=7$.

\begin{lem}\label{lemmah13}
 Let $G$ be a connected graph and let $P$ be a 2-path   disjoint with $G$. Suppose $G'=G\cup P+e$ with $e$ being an edge connecting $G$ and $P$. Then $$diss(G')=diss(G)+2.$$
 \end{lem}
 \begin{proof}
  Note that a maximum dissociation set of $G'$ contains at most two vertices from $P$. we have $diss(G')\le diss(G)+2$. On the other hand, suppose $S$ is a maximum dissociation set of $G$ and $v$ is the vertex in $V(P)$ incident with $e$. Then $S\cup \left(V(P)\setminus\{v\}\right)$ is a  dissociation set of $G'$. Therefore, we have $diss(G')\ge diss(G)+2$, and hence $diss(G')= diss(G)+2$.
 \end{proof}

\begin{lem}
	\label{lemma8}
Let $n\ge 3$ be an integer. Then   $$diss(P_n)=diss(W_n)=diss(\tilde{W}_{n})=\left\lceil\frac{2n}{3}\right\rceil\quad  and \quad diss(C_n)= \left\lfloor\frac{2n}{3} \right\rfloor.$$
\end{lem}
\begin{proof}
Note that $S$ is a dissociation set of a graph $G$ if and only if $V(G)\setminus S$ is a vertex 3-path cover of $G$.
By Proposition 1.1 of \cite{BRESAR20131943},   we have $diss(P_n)=\left\lceil {2n}/{3}\right\rceil$ and $diss(C_n)= \left\lfloor {2n}/{3} \right\rfloor.$

Notice that $W_n$ is obtained from $P_{n-3}+P_3$ by adding an edge between the center of $P_3$ and an end of $P_{n-3}$. Applying Lemma \ref{lemma13} we have
$$diss(W_n)=diss(P_{n-3})+2=\left\lceil\frac{2n}{3}\right\rceil.$$
 Similarly, since $\tilde{W}_{n}$ is obtained from $W_{n-3}+P_3$ by adding an edge, we have $diss(\tilde{W}_{n})=\left\lceil {2n}/{3}\right\rceil.$
\end{proof}
An {\it internal path} of a graph $G$ is a sequence of vertices $v_1, v_2,\ldots, v_k$ with $k\ge2$ such that:\\
	 \indent (1) the vertices in the sequence are distinct(except possibly $v_1=v_k$);\\
	 \indent (2) $v_i$ is adjacent to $v_{i+1}(i=1,2,\ldots,k-1)$;\\
	\indent (3) the vertex degrees satisfy $d(v_1)\ge3$, $d(v_2)=\cdots=d(v_{k-1})=2$ (unless $k=2$) and \indent $d(v_k\ge3)$.

Let $uv$ be an edge of a graph $G$. The {\it subdivision} of the edge $uv$ is replacing $uv$ with a 2-path, i.e., deleting $uv$ and  adding  a new vertex $w$ and two new edges $uw$, $wv$. Generally,  the {\it $k$-subdivision} of  $uv$ is replacing $uv$ with a $(k+1)$-path.
We denote by  $G_{uv}^{(k)} $ the graph  obtained from $G$ by doing a $k$-subdivision of $uv$ for $k=1,2,\ldots$, where $G_{uv}^{(1)} $ will be simplified  to be $G_{uv}$.   We  also use   $G_{uv}(w)$  and $G_{uv}^{(k)}(w_1,\ldots,w_k)$ to denote $G_{uv} $ and $G_{uv}^{(k)}$ if we need to emphasize the new added vertices $w$ and $w_1,\ldots,w_k$, respectively.

\begin{lem}
	\label{lemma12}\cite{0} Suppose that $G\ncong \tilde{W}_{n}$ and $uv$ is an edge on an internal path of $G$.  Then $\rho(G_{uv})<\rho(G)$.
\end{lem}
	
\begin{lem}
\label{lemma14}	Let $uv$ be an edge of a  tree $T$.   Then
	\begin{eqnarray}
diss(T_{uv})\in\{diss(T),~diss(T)+1\},\label{eqh1} \\
	 diss(T^{(2)}_{uv})\in\{diss(T)+1,~diss(T)+2\}\label{eqh2}
\end{eqnarray}
and
\begin{equation}\label{eqh3}
	 diss(T^{(3)}_{uv})= diss(T)+2.
\end{equation}
\end{lem}
\begin{proof}
Suppose $T_{uv}=T_{uv}(w)$. It is clear that  $diss(T_{uv})\ge diss(T)$.
 Let $S $ be a maximum dissociation set of $T_{uv}(w)$. If $w\notin S $, then $S \backslash\{u\}$ is a dissociation set of $T$, and thus $diss(T)\ge|S \backslash\{u\}|\ge diss(T_{uv})-1$. If $w \in S $, then at least one of $u$ and $v$ is not in $S $. Thus $S \backslash\{w\}$ is a dissociation set of $T$, which also leads to $diss(T)\ge|S \backslash\{w\}|=diss(T_{uv})-1$. In both cases we have $$diss(T_{uv})\le diss(T)+1.$$  Therefore, we have (\ref{eqh1}).

 Suppose $T_{uv}^{(2)}=T_{uv}^{(2)}(w_1,w_2)$. Let $S$ be a maximum dissociation set of $T$. It is easy to verify $$diss(T_{uv}^{(2)}(w_1,w_2))\ge diss(T)+1,$$
 since we can always add one of $w_1$ and $w_2$ to $S$ to get a dissociation set of $T_{uv}^{(2)}{(w_1,w_2)}$. On the other hand,
  suppose $S'$ is a maximum dissociation of $T_{uv}^{(2)}(w_1,w_2)$.   Then $S'-\{u,w_1,w_2\} $ is a dissociation set of $T$ and
 $$diss(T_{uv}^{(2)}(w_1,w_2))- 2\le |S'-\{u,w_1,w_2\} |\le diss(T).$$
Therefore, we have (\ref{eqh2}).

 Suppose $T_{uv}^{(3)}=T_{uv}^{(3)}(w_1,w_2, w_3)$. Given a  maximum dissociation set $S$ of $T$, if $\{u,v\}\subseteq S$, then $S\cup \{w_1,w_3\}$ is a dissociation set of $T_{uv}^{(3)}(w_1,w_2, w_3)$; if at least one of $u$ and $v$ is not in $S$, say $u\not \in S$, then $S\cup\{w_1,w_2\}$ is a dissociation set of $T_{uv}^{(3)}(w_1,w_2, w_3)$. Therefore, we have
 $$diss(T_{uv}^{(3)}(w_1,w_2, w_3))\ge diss(T)+2.$$
 On the other hand, suppose $S'$ is a maximum dissociation set of $T_{uv}^{(3)}(w_1,w_2, w_3)$. If  $\{u,v,w_2\}\subseteq S'$, then $S'-\{u,w_2\}$ is a dissociation set of $T$, which leads to  $$diss(T_{uv}^{(3)}(w_1,w_2, w_3))-2=|S'|-2=|S'-\{u,w_2\}|\le diss(T).$$
 If  $\{u,v\}\subseteq S'$ and $w_2\not\in S'$, let $N(u)$ and $N(v)$ be the sets of neighbors of $u$ and $v$ in $T_{uv}^{(3)}(w_1,w_2, w_3)$, respectively. Then $S'-N(u)\cup N(v)$ is a dissociation set of $T$, which leads to
  $$diss(T_{uv}^{(3)}(w_1,w_2, w_3))-2=|S'|-2=|S'-N(u)\cup N(v)|\le diss(T),$$
  since  $|N(u)\cap S'|=|N(v)\cap S'|=1$.
 If $|\{u,v\}\cap S'|\le 1$, then $S'-\{w_1,w_2,w_3\}$ is a dissociation set of $T$, which leads to   $$diss(T_{uv}^{(3)}(w_1,w_2, w_3))-2=|S'|-2=|S'-\{w_1,w_2,w_3\}|\le diss(T),$$
 since $|S'\cap \{w_1,w_2,w_3\}|=2$. In all cases we have
 $$diss(T_{uv}^{(3)}(w_1,w_2, w_3))\le diss(T)+2.$$
 Therefore, we have (\ref{eqh3}).
 \end{proof}

\section{The proof of Theorem \ref{th2}}

Denote by  $B(n, s,t)$ the graph of order $n$ obtained from a path $P_{n-s-2t}$ by attaching $s$ leaves and $t$ edges to the same end of  $P_{n-s-2t}$, which has the following diagram.   \begin{figure}[H]
	\centering
	\begin{tikzpicture}
	\draw(0,0)[fill]circle[radius=1.0mm]--(1,0)[fill]circle[radius=1.0mm];
	\draw(2,0)[fill]circle[radius=1.0mm]--(1,0)[fill]circle[radius=1.0mm];
	\draw(3,0)[fill]circle[radius=1.0mm]--(4,0)[fill]circle[radius=1.0mm];
	\draw(5,0)[fill]circle[radius=1.0mm]--(4,0)[fill]circle[radius=1.0mm];
	\draw(5,0)[fill]circle[radius=1.0mm]--(6,-0.5)[fill]circle[radius=1.0mm];
	\draw(7,-1)[fill]circle[radius=1.0mm]--(6,-0.5)[fill]circle[radius=1.0mm];
	\draw(5,0)[fill]circle[radius=1.0mm]--(6,0)[fill]circle[radius=1.0mm];
	\draw(5,0)[fill]circle[radius=1.0mm]--(4.5,0.9)[fill]circle[radius=1.0mm];
	\draw(5,0)[fill]circle[radius=1.0mm]--(5.5,0.9)[fill]circle[radius=1.0mm];
	\node at (2.5,0){$\cdots$};
	\draw(6,0)[fill]circle[radius=1.0mm]--(7,0)[fill]circle[radius=1.0mm];
	\node at (5,0.9){$\cdots$};
	\node at (5,1.5){$s$};
	\node at (7,-0.43){$\vdots$};
	\draw[decorate,decoration={brace,mirror,amplitude=0.2cm}](5.5,0.99)--(4.5,0.99);
	\draw[decorate,decoration={brace,mirror,amplitude=0.2cm}](7.1,-1)--(7.1,0);
	\node at (7.4,-0.5){$t$};
	\end{tikzpicture}
	\label{fig:5.5}
	\caption{The graph $B(n, s,t)$}
\end{figure}
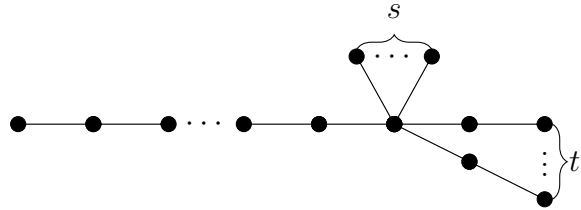
If $s+t=1$, then $B(n, s,t)$ is a path with  dissociation number $\lceil 2n/3\rceil$. If $s+t\ge 2$, since there is a maximum dissociation set of $B(n, s,t)$ that does not contain the unique branch vertex,   we have
 \begin{equation}\label{eqh4}
diss(B(n,s,t))=s+2t+diss(P_{n-s-2t-1})=s+2t+\lceil2(n-s-2t-1)/3\rceil.
\end{equation}

Now we adopt a similar scheme as the proof of Theorem 1.2 in \cite{LOU2022112778} to prove Theorem \ref{th2}.

{\it Proof of Theorem \ref{th2}.}
Suppose $G$ attains the minimum spectral radius in $\mathcal{G}_{n,k}$ and it is not a tree. Then $G$ contains a spanning tree $T$.  Applying Lemma \ref{lemma3}, we have $\rho(T)<\rho(G).$ Since $G$ attains the  minimum spectral radius  in $\mathcal{G}_{n,k}$, we have
 \begin{equation*}\label{eqh5}
	 diss(T)> diss(G)>\left\lceil\frac{2n}{3}\right\rceil.
	\end{equation*}
Thus $T$ is not a path, recalling that $diss(P_n)=\lceil 2n/3\rceil$.
Now we distinguish  two cases.

{\it	Case 1.} $T$ has exactly one branch vertex $u_0$.  Suppose $d_{T}(u_0)=k$ and $P^{(1)},\ldots,P^{(k)}$ are the $k$ branch paths attached to $u_0$. Without loss of generality, we assume $P^{(1)}$ is the longest branch path.   If $P^{(i)}$ has length larger than 1 for some $i\ge 2$, by Lemma \ref{lemma2} and Lemma \ref{lemmah13}, we can cut a $P_3$ from $P^{(i)}$ and attach it to $P^{(1)}$ to obtain a tree with the same dissociation number and smaller spectral radius. Repeating this process until all but one branch paths have length less than or equal to 1, we can find a tree $T'=B(n,s,t)$  such that
\begin{equation}\label{eqh6}
|V(T')|=|V(T)|, \quad diss(T')=diss(T)\quad\text{ and}\quad \rho(T')\le \rho(T).
\end{equation}
Moreover,
since $diss(T')=diss(T)>\lceil 2n/3\rceil+1$, we have $s +2t\ge 6.$

Let $s_1=s,t_1=t$. We construct a tree sequence $T_{i}=B(n,s_i,t_i)$, $i=1,2,\ldots$, as follows:

(i) if $s_i=0$, let $t_{i+1}=t_{i}-2,~ s_{i+1}=s_{i}+1$;

(ii) if $s_i=1$, let $t_{i+1}=t_i-1, ~s_{i+1}=s_i-1$;

(iii) if $s_i\ge 2$, let $t_{i+1}=t_i+1,~ s_{i+1}=s_i-2$.

Then by (\ref{eqh4}) and Lemma \ref{lemma2}, we have
\begin{equation}\label{eqh7}
diss(T_{i+1})\in \{diss(T_i), diss(T_i)-1\}\quad \text{ and}\quad \rho(T_{i+1})<\rho(T_i).
\end{equation}
In the above tree sequence, there exists a tree $T_k=B(n,s_k,t_k)$ such that $diss(T_k)=diss(G)$, since $diss(T_1)>diss(G)$ and $diss(B(n,s,t))<diss(G)$ when $2t+s<6$. On the other hand,  by (\ref{eqh6}) and (\ref{eqh7}), we have $\rho(T_k)<\rho(G)$, which contradicts the assumption on $G$.

 {\it Case 2.}  Suppose $T$ has  at least two branch vertices.   Let $T_1=T$. For $i=1,2,\ldots,$ if $T_i$ has two branch  vertices, we construct the tree $T_{i+1}$ as follows.

 (i) If $T_{i}$ has a branch path $P_k$ with $k\ge 3$,   $T_{i+1}$ is obtained from $T_i$ by doing a  $3$-subdivision of an internal edge, and then  deleting a pendant $P_3$ from the branch path $P_k$.

 (ii) If all branch paths in $T_i$ have lengths 1 or 2, we distinguish two cases. If there is an end branch vertex attached with $s$  leaves $v_1,\ldots,v_s$ and $t$ disjoint edges $u_1w_1,\ldots,u_tw_t$ such that $s+2t\ge 3$,   $T_{i+1}$ is obtained  from $T_i$ by doing  a   subdivision of an internal edge, and then deleting a leaf from $\{v_1,\ldots,v_s,u_1,\ldots,u_t,w_1,\ldots,w_t\}$. If every end branch vertex is only attached with two leaves, $T_{i+1}$ is  obtained  from $T_i$ by doing a 3-subdivision of an internal edge, and then deleting an end branch vertex together with the leaves attached to it.

Suppose $k$ is the smallest integer such that $diss(T_k)=diss(G)$ or $T_k$ has exactly one branch vertex. Then
applying Lemma \ref{lemma14}, we have
$$diss(T_{i+1})\in\{diss(T_i),diss(T_i)-1\} \quad \text{for}\quad i=1,2,\ldots,k-1.$$
Since $diss(T_i)>\lceil{2n/3}\rceil$, we have $T_i\not\cong \tilde{W}_n$ for $i=1,\ldots,k-1$. Applying
Lemma \ref{lemma3} and Lemma \ref{lemma12}, we have
  $$  \rho(T_{i+1})<\rho(T_i)\quad \text{for}\quad i=1,2,\ldots,k-1.$$

  If $diss(T_k)= diss(G)$,  then $\rho(T_{k})<\rho(T_1)<\rho(G)$ contradicts the assumption on $G$. If $diss(T_k)> diss(G)$, then using the same construction as in Case 1, we can always find a tree $T''$ such that $diss(T'')=diss(G)$ and $\rho(T'')<\rho(G)$, which contradicts the assumption on $G$.

Therefore, $G$ is a tree. This completes the proof.\hspace{7cm} $\square$\\

\par

From the above proof we have the following.
\begin{corollary}
Let $ \lceil 2n/3\rceil< k\le n-1$. Suppose a tree $T$ attains the minimum spectral radius in  $\mathcal{G}_{n,k}$. Then either every branch path in $T$ has length  at most one or $T\in B(n,s,t)$ for some positive integers $s,t$.
\end{corollary}

 \section{Proof  of Theorem \ref{th1}}
 In this section, we present the proof of Theorem \ref{th1}. We divide the proof into two parts. In the first part we discuss the cases   $k\in\{n-1,~\lceil\frac{2n}{3}\rceil,~\lfloor\frac{2n}{3}\rfloor,~2\}$, and in the second part we discuss the case   $k=n-2$.
 \subsection{The cases  $k\in\{n-1,~\lceil\frac{2n}{3}\rceil,~\lfloor\frac{2n}{3}\rfloor,~2\}$}

Similarly to the proof of Theorem 2.2 and Theorem 2.5 in \cite{XU2009937}, the cases for  $k=\lceil2n/3\rceil$ and $k=\lfloor 2n/3\rfloor$ can be concluded from Lemma \ref{lemma1}, Lemma \ref{lemma6}, Lemma \ref{lemma7} and Lemma \ref{lemma8} directly.

\par
{\it Proof of the case $k=2$.}~~
Suppose $k=2$. 	Since $diss(K_n-e)=2$ for every edge $e$ in $K_n$,  $G$ contains at least two  nonadjacent vertices. Given an arbitrary pair of nonadjacent vertices $u,v$ in $G$,  since $diss(G)=2$, both $u$ and $v$ are adjacent to all the other vertices in $G$.

Suppose $G$ has exactly $k$ disjoint pairs of  nonadjacent vertices $(u_i,v_i),i=1,...,k$. Then $G=\{u_1,v_1\}\vee\{u_2,v_2\}\vee\cdots \vee\{u_k,v_k\}\vee K_{n-2k}$. If $n-2k\ge 2$, then deleting an edge  $uv$ from  $G$ with $u,v\in V(G)\setminus \{u_1,\ldots,u_k,v_1,\ldots,v_k\}$, we get a subgraph $H$ of $G$ with $\rho(H)<\rho(G)$ and $diss(H)=2$, a contradiction. Therefore, $n-2k<2$ and $G$ is  a balanced complete $\lceil n/2\rceil$-partite graph.   {\hspace{12cm} $\square$}

{\it Proof of the case $k=n-1$.}~~
Suppose $k=n-1$.  By Theorem \ref{th2}, $G$ is a tree.	  Let $S$ be a maximum dissociation set of $G$ such that $G[S]$ has exactly $r$ isolated vertices and $t$ disjoint edges with $r+2t=n-1$. Denote  by $v$ the unique vertex in $ V(G)\setminus S$. Since $G$ is a tree, $v$ is adjacent to exactly one end of each edge from $G[S]$, and $v$ is adjacent to all isolated vertices of $G[S]$.
  Therefore, we  have $G\cong S(r,t)$.
 Applying Lemma \ref{lemma2}, we have $G\cong S(0,\frac{n-1}{2})$ when $n$ is odd, and $G\cong S(1,\frac{n-2}{2})$ when $n$ is even.  {\hspace{7.8cm} $\square$}

 \subsection{The case   $k=n-2$}

 Suppose $G$ attains the minimum spectral radius in $\mathcal{G}_{n,n-2}$. If $n=5$, then $n-2=\lfloor 2n/3\rfloor$ and $G\cong C_5$. If $n\in\{6,7,8\}$, $n-2=\lceil 2n/3\rceil$ and $G\cong P_n$. If $n=9$, applying Lemma \ref{lemma7}, we have $G\cong \tilde{E}_8$, since $diss(\tilde{E}_8)=7$ and $diss(P_9)=6$. We present the proof for the case $n\ge 10$ of Theorem \ref{th1} in this subsection.

Denote by  $G_1(r,s,p,q),\ldots,G_4(r,s,p,q)$  the following graphs (see Figure \ref{fig:11}):
\begin{itemize}
\item[(i)] $G_1(r,s,p,q)$: the graph obtained from an edge $v_1v_2$ by attaching $r$  leaves and $s$ disjoint edges  to $v_1$, and attaching $p$  leaves and $q$ disjoint  edges  to $v_2$;
\item[(ii)] $G_2(r,s,p,q)$: the graph obtained from a path $v_1v_3v_2$ by attaching $r$   leaves and $s$ disjoint  edges  to $v_1$, and attaching  $p$   leaves and $q$ disjoint  edges  to $v_2$;
\item[(iii)] $G_3(r,s,p,q)$: the  graph obtained from a path $v_1v_3v_4v_2$ by attaching $r$   leaves and $s$ disjoint  edges  to $v_1$, and attaching $p$  leaves and $q$ disjoint  edges  to $v_2$;
\item[(iv)] $G_4(r,s,p,q)$: the graph obtained from a path $v_1v_3v_2$ by attaching $r$   leaves and $s$ disjoint  edges  to $v_1$,   attaching $p$   leaves and $q$ disjoint  edges  to $v_2$ and attaching a  vertex $v_4$ to $v_3$.
\end{itemize}
	\begin{figure}[H]
	\centering
	\begin{subfigure}{.51\textwidth}
		\centering
		\begin{tikzpicture}
		\foreach \y in {2}{
			\foreach \x in {3.5,4,6,6.5}{
				\draw (\x, \y)[fill]circle[radius=1.0mm];
				
			}
		};
		\draw(4,3)[fill]circle[radius=1.0mm]--(6,3)[fill]circle[radius=1.0mm];
		\node at (4,3)[above=3pt]{$v_1$};
		\node at (6,3)[above=3pt]{$v_2$};
		\draw(4,3)[fill]circle[radius=1.0mm]--(3,3)[fill]circle[radius=1.0mm];
		\draw(3,2.3)[fill]circle[radius=1.0mm]--(4,3)[fill]circle[radius=1.0mm];
		\draw(4,3)[fill]circle[radius=1.0mm]--(4,2)[fill]circle[radius=1.0mm];
		\draw(6,3)[fill]circle[radius=1.0mm]--(7,3)[fill]circle[radius=1.0mm];
		\draw(7,2.3)[fill]circle[radius=1.0mm]--(6,3)[fill]circle[radius=1.0mm];
		\node at (3,2.75){$\vdots$};
		\node at (7,2.75){$\vdots$};
		\node at(2.5,3)[below=3pt]{$r$};
		\node at(7.5,3)[below=3pt]{$p$};
		\draw[decorate,decoration={brace,mirror,amplitude=0.2cm}](7.14,2.3)--(7.14,3);
		\draw[decorate,decoration={brace,mirror,amplitude=0.2cm}](2.86,3)--(2.86,2.3);
		\draw(6.5,2)[fill]circle[radius=1.0mm]--(7,1)[fill]circle[radius=1.0mm];
		\foreach \y in {1}{
			\foreach \x in {3,4,7,6}{
				\draw (\x, \y)[fill]circle[radius=1.0mm];
				
			}
		};
		\foreach \y in{2}{
			\foreach \x in {3.5,4}{
				\draw(\x,\y)--(4,3);
			}
		};
		\draw(3.5,2)[fill]circle[radius=1.0mm]--(3,1)[fill]circle[radius=1.0mm];
		\foreach \y in{2}{
			\foreach \x in {6,6.5}{
				\draw(\x,\y)--(6,3);
			}
		};
		\foreach \y in{1}{
			\foreach \x in {4,6}{
				\draw(\x,\y)--(\x,2);}};
		\draw[decorate,decoration={brace,mirror,amplitude=0.2cm}](3,0.85)--(4,0.85);
		\node at(3.5,0.8)[below=3pt]{$s$};
		\draw[decorate,decoration={brace,mirror,amplitude=0.2cm}](6,0.85)--(7,0.85);
		\node at(6.5,0.8)[below=3pt]{$q$};
		\node at (3.5,1){$\cdots$};
		\node at ((6.5,1){$\cdots$};
		\end{tikzpicture}
		\caption*{$G_1(r,s,p,q)$}
		\label{fig:11.1}
	\end{subfigure}
	\begin{subfigure}{.48\textwidth}
	\centering
	\begin{tikzpicture}
	\foreach \y in {2}{
	\foreach \x in {3.5,4,6,6.5}{
		\draw (\x, \y)[fill]circle[radius=1.0mm];
		
	}
};
\draw(6.5,2)[fill]circle[radius=1.0mm]--(7,1)[fill]circle[radius=1.0mm];
\draw(6.5,2)[fill]circle[radius=1.0mm]--(6,3)[fill]circle[radius=1.0mm];
\draw(3.5,2)[fill]circle[radius=1.0mm]--(3,1)[fill]circle[radius=1.0mm];
\draw(4,2)[fill]circle[radius=1.0mm]--(4,3)[fill]circle[radius=1.0mm];
\draw(4,3)[fill]circle[radius=1.0mm]--(6,3)[fill]circle[radius=1.0mm];
\draw(5,3)[fill]circle[radius=1.0mm];
\node at (4,3)[above=3pt]{$v_1$};
\node at (6,3)[above=3pt]{$v_2$};
\draw(4,3)[fill]circle[radius=1.0mm]--(3,3)[fill]circle[radius=1.0mm];
\draw(3,2.3)[fill]circle[radius=1.0mm]--(4,3)[fill]circle[radius=1.0mm];
\draw(6,3)[fill]circle[radius=1.0mm]--(7,3)[fill]circle[radius=1.0mm];
\draw(7,2.3)[fill]circle[radius=1.0mm]--(6,3)[fill]circle[radius=1.0mm];
\node at (3,2.75){$\vdots$};
\node at (7,2.75){$\vdots$};
	\node at(2.5,3)[below=3pt]{$r$};
\node at(7.5,3)[below=3pt]{$p$};
\draw[decorate,decoration={brace,mirror,amplitude=0.2cm}](7.14,2.3)--(7.14,3);
\draw[decorate,decoration={brace,mirror,amplitude=0.2cm}](2.86,3)--(2.86,2.3);
\foreach \y in {1}{
	\foreach \x in {3,4,7,6}{
		\draw (\x, \y)[fill]circle[radius=1.0mm];
		
	}
};
\foreach \y in{2}{
	\foreach \x in {3.5,...,4}{
		\draw(\x,\y)--(4,3);
	}
};
\foreach \y in{2}{
	\foreach \x in {6,...,6.5}{
		\draw(\x,\y)--(6,3);
	}
};
\foreach \y in{1}{
	\foreach \x in {4,6}{
		\draw(\x,\y)--(\x,2);}};
\draw[decorate,decoration={brace,mirror,amplitude=0.2cm}](3,0.85)--(4,0.85);
\node at(3.5,0.8)[below=3pt]{$s$};
\draw[decorate,decoration={brace,mirror,amplitude=0.2cm}](6,0.85)--(7,0.85);
\node at(6.5,0.8)[below=3pt]{$q$};
\node at (3.5,1){$\cdots$};
\node at ((6.5,1){$\cdots$};
\node at (5,3)[above=3pt]{$v_3$};
	\end{tikzpicture}
	\caption*{$G_2(r,s,p,q)$}
	\label{fig:11.2}
\end{subfigure}
	\begin{subfigure}{.51\textwidth}
		\centering
		\vspace{23pt}
		\begin{tikzpicture}
		\foreach \y in {2}{
		\foreach \x in {3.5,4,6,6.5}{
			\draw (\x, \y)[fill]circle[radius=1.0mm];
			
		}
	};
	\draw(4,3)[fill]circle[radius=1.0mm]--(6,3)[fill]circle[radius=1.0mm];
	\draw(6.5,2)[fill]circle[radius=1.0mm]--(7,1)[fill]circle[radius=1.0mm];
	\draw(6.5,2)[fill]circle[radius=1.0mm]--(6,3)[fill]circle[radius=1.0mm];
	\draw(3.5,2)[fill]circle[radius=1.0mm]--(3,1)[fill]circle[radius=1.0mm];
	\draw(4,2)[fill]circle[radius=1.0mm]--(4,3)[fill]circle[radius=1.0mm];
	\node at (4,3)[above=3pt]{$v_1$};
	\node at (6,3)[above=3pt]{$v_2$};
		\draw(4,3)[fill]circle[radius=1.0mm]--(3,3)[fill]circle[radius=1.0mm];
	\draw(3,2.3)[fill]circle[radius=1.0mm]--(4,3)[fill]circle[radius=1.0mm];
	\draw(6,3)[fill]circle[radius=1.0mm]--(7,3)[fill]circle[radius=1.0mm];
	\draw(7,2.3)[fill]circle[radius=1.0mm]--(6,3)[fill]circle[radius=1.0mm];
	\node at (3,2.75){$\vdots$};
	\node at (7,2.75){$\vdots$};
	\node at(2.5,3)[below=3pt]{$r$};
	\node at(7.5,3)[below=3pt]{$p$};
	\draw[decorate,decoration={brace,mirror,amplitude=0.2cm}](7.14,2.3)--(7.14,3);
	\draw[decorate,decoration={brace,mirror,amplitude=0.2cm}](2.86,3)--(2.86,2.3);
	\foreach \y in {1}{
		\foreach \x in {3,4,7,6}{
			\draw (\x, \y)[fill]circle[radius=1.0mm];
			
		}
	};
	\foreach \y in{2}{
		\foreach \x in {3.5,...,4}{
			\draw(\x,\y)--(4,3);
		}
	};
	\foreach \y in{2}{
		\foreach \x in {6,...,6.5}{
			\draw(\x,\y)--(6,3);
		}
	};
	\foreach \y in{1}{
		\foreach \x in {4,6}{
			\draw(\x,\y)--(\x,2);}};
	\draw[decorate,decoration={brace,mirror,amplitude=0.2cm}](3,0.85)--(4,0.85);
	\node at(3.5,0.8)[below=3pt]{$s$};
	\draw[decorate,decoration={brace,mirror,amplitude=0.2cm}](6,0.85)--(7,0.85);
	\node at(6.5,0.8)[below=3pt]{$q$};
	\node at (3.5,1){$\cdots$};
	\node at ((6.5,1){$\cdots$};
			\node at (4.7,3)[above=3pt]{$v_3$};
		\node at (5.35,3)[above=3pt]{$v_4$};
			\draw(4.7,3)[fill]circle[radius=1.0mm];
		\draw(5.35,3)[fill]circle[radius=1.0mm];
		\end{tikzpicture}
		\caption*{$G_3(r,s,p,q)$}
		\label{fig:11.3}
	\end{subfigure}
	\begin{subfigure}{.48\textwidth}
		\centering
		\vspace{23pt}
		\begin{tikzpicture}
		\foreach \y in {2}{
			\foreach \x in {4,6}{
				\draw (\x, \y)[fill]circle[radius=1.0mm];
				
			}
		};
		\draw(4,3)[fill]circle[radius=1.0mm]--(6,3)[fill]circle[radius=1.0mm];
		\draw(6.5,2)[fill]circle[radius=1.0mm]--(7,1)[fill]circle[radius=1.0mm];
		\draw(6.5,2)[fill]circle[radius=1.0mm]--(6,3)[fill]circle[radius=1.0mm];
		\draw(3.5,2)[fill]circle[radius=1.0mm]--(3,1)[fill]circle[radius=1.0mm];
		\draw(4,2)[fill]circle[radius=1.0mm]--(4,3)[fill]circle[radius=1.0mm];
	\node at (4,3)[above=3pt]{$v_1$};
	\node at (6,3)[above=3pt]{$v_2$};
	\draw(4,3)[fill]circle[radius=1.0mm]--(3,3)[fill]circle[radius=1.0mm];
	\draw(3,2.3)[fill]circle[radius=1.0mm]--(4,3)[fill]circle[radius=1.0mm];
	\draw(6,3)[fill]circle[radius=1.0mm]--(7,3)[fill]circle[radius=1.0mm];
	\draw(7,2.3)[fill]circle[radius=1.0mm]--(6,3)[fill]circle[radius=1.0mm];
	\node at (3,2.75){$\vdots$};
	\node at (7,2.75){$\vdots$};
	\node at(2.5,3)[below=3pt]{$r$};
	\node at(7.5,3)[below=3pt]{$p$};
	\draw[decorate,decoration={brace,mirror,amplitude=0.2cm}](7.14,2.3)--(7.14,3);
	\draw[decorate,decoration={brace,mirror,amplitude=0.2cm}](2.86,3)--(2.86,2.3);
		\foreach \y in {1}{
			\foreach \x in {3,4,7,6}{
				\draw (\x, \y)[fill]circle[radius=1.0mm];
				
			}
		};
		\foreach \y in{2}{
			\foreach \x in {3.5,...,4}{
				\draw(\x,\y)--(4,3);
			}
		};
		\foreach \y in{2}{
			\foreach \x in {6,...,6.5}{
				\draw(\x,\y)--(6,3);
			}
		};
		\foreach \y in{1}{
			\foreach \x in {4,6}{
				\draw(\x,\y)--(\x,2);}};
		\draw[decorate,decoration={brace,mirror,amplitude=0.2cm}](3,0.85)--(4,0.85);
		\node at(3.5,0.8)[below=3pt]{$s$};
		\draw[decorate,decoration={brace,mirror,amplitude=0.2cm}](6,0.85)--(7,0.85);
		\node at(6.5,0.8)[below=3pt]{$q$};
		\node at (3.5,1){$\cdots$};
		\node at ((6.5,1){$\cdots$};
		\node at (5,2)[below=3pt]{$v_4$};
		\node at (5,3)[above=3pt]{$v_3$};
		\draw(4,3)[fill]circle[radius=1.0mm]--(5,3)[fill]circle[radius=1.0mm];
		\draw(5,3)[fill]circle[radius=1.0mm]--(6,3)[fill]circle[radius=1.0mm];
		\draw(5,3)[fill]circle[radius=1.0mm]--(5,2)[fill]circle[radius=1.0mm];
		\end{tikzpicture}
		\caption*{$G_4(r,s,p,q)$}
		\label{fig:11.4}
	\end{subfigure}
	\caption{}
	\label{fig:11}
\end{figure}

Firstly we prove the following claims.

\noindent\textbf{Claim 1.} \emph{Let $k>0$ be an integer and $T^*_k=S(0,k)$.  Then $\rho(T^*_k)=\sqrt{k+1}$.}
	
{\it Proof.}
Suppose $u_1$ is the center of  $T^*_k$ with neighbors  $u_2,\ldots, u_{k+1}$.  We label the leaf adjacent to $u_i$ by $u_{k+i}$ for $i=2,\ldots, k+1$. Let $Z=(z_1,z_2,...,z_n)^T$ be the Perron vector of $T^*_k$ with $z_i$ corresponding to $u_i$. By  symmetry of the components of $Z$ and $\rho(T^*_k)Z=A(T^*_k)Z$,  we have
\begin{equation*} \label{equa4}
        \rho(T^*_k)z_{1}=kz_{2},\quad
        \rho(T^*_k)z_{2}=z_{1}+z_{k+2}, \quad
        \rho(T^*_k)z_{k+2}=z_{2},
\end{equation*}
which lead to $$\rho(T^*_k)z_1=k\rho(T^*_k)z_{k+2}=k\rho(T^*_k)[\rho(T^*_k)z_2-z_1]=\rho^3(T^*_k)z_1-k\rho(T^*_k)z_1.$$ Hence, $\rho(T^*_k)=\sqrt{k+1}$.\\
	
\par

\noindent\textbf{Claim 2. }\emph{ Let $s $ and $q$ be integers such that  $s+1 \geq q \geq2$. Then \begin{equation}\label{eqc2}
\rho(H(2s+2q+4))\le\rho(G_3(0,s,0,q))<\rho(G_3(1,s,1,q-1)),
\end{equation}
where  equality in the left inequality holds if and only if $G_3(0,s,0,q)\cong H(2s+2q+4)$.
 }

{\it Proof.}
 For convenience, suppose $L=G_3(0,s,0,q)$ and $R=G_3(1,s,1,q-1)$ with their vertices  labelled as follows.
\begin{figure}[H]
	\centering
	\begin{subfigure}{.45\textwidth}
		\centering
		\begin{tikzpicture}
			\draw[fill=black](0,0)circle(0.1)
			(0.6,0)circle(0.1)
			(1.2,0)circle(0.1)
			(1.8,0)circle(0.1)
			(1.8,-0.7)circle(0.1)
			(1.8,-1.4)circle(0.1)
			(2.6,-0.7)circle(0.1)
			(2.6,-1.4)circle(0.1)
			(0,-0.7)circle(0.1)
			(0,-1.4)circle(0.1)
			(-0.8,-0.7)circle(0.1)
			(-0.8,-1.4)circle(0.1);
			\draw(0,0)--(1.8,0)
			(0,0)--(0,-1.4)
			(0,0)--(-0.8,-0.7)
			(-0.8,-0.7)--(-0.8,-1.4)
			(1.8,0)--(1.8,-1.4)
			(1.8,0)--(2.6,-0.7)
			(2.6,-0.7)--(2.6,-1.4);		
			\draw[decorate,decoration={brace,mirror,amplitude=0.2cm}](-0.8,-1.55)--(0,-1.55);
			\draw[decorate,decoration={brace,mirror,amplitude=0.2cm}](1.8,-1.55)--(2.6,-1.55);
			\node at (-0.4,-1.95){$s$};
			\node at (2.2,-1.95){$q$};
			\node at (-0.4,-1.3){$\cdots$};
			\node at (2.2,-1.3){$\cdots$};
			\node at (-1.2,-1.4){$v_6$};
			\node at (-1.2,-0.7){$v_5$};
			\node at (0,0.25){$v_1$};
			\node at (0.6,0.25){$v_3$};
			\node at (1.2,0.25){$v_4$};
			\node at (1.8,0.25){$v_2$};
			\node at (2.95,-0.7){$v_7$};
			\node at (2.95,-1.4){$v_8$};
		\end{tikzpicture}
		\caption*{$L=G_3(0,s,0,q)$}
		\label{fig:12.2}
	\end{subfigure}
	\begin{subfigure}{.45\textwidth}
		\centering
		\begin{tikzpicture}
			\draw[fill=black](0,0)circle(0.1)
			(0.6,0)circle(0.1)
			(1.2,0)circle(0.1)
			(1.8,0)circle(0.1)
			(1.8,-0.7)circle(0.1)
			(1.8,-1.4)circle(0.1)
			(2.6,-0.7)circle(0.1)
			(2.6,-1.4)circle(0.1)
			(0,-0.7)circle(0.1)
			(0,-1.4)circle(0.1)
			(-0.8,-0.7)circle(0.1)
			(-0.8,-1.4)circle(0.1)
			(0,0.7)circle(0.1)
			(1.8,0.7)circle(0.1);
			\draw(0,0)--(1.8,0)
			(0,0)--(0,-1.4)
			(0,0)--(-0.8,-0.7)
			(-0.8,-0.7)--(-0.8,-1.4)
			(1.8,0)--(1.8,-1.4)
			(1.8,0)--(2.6,-0.7)
			(2.6,-0.7)--(2.6,-1.4)
			(0,0)--(0,0.7)
			(1.8,0)--(1.8,0.7);		
			\draw[decorate,decoration={brace,mirror,amplitude=0.2cm}](-0.8,-1.55)--(0,-1.55);
			\draw[decorate,decoration={brace,mirror,amplitude=0.2cm}](1.8,-1.55)--(2.6,-1.55);
			\node at (-0.4,-1.95){$s$};
			\node at (2.2,-1.95){$q-1$};
			\node at (-0.4,-1.3){$\cdots$};
			\node at (2.2,-1.3){$\cdots$};
			\node at (-1.2,-1.4){$v_6$};
			\node at (-1.2,-0.7){$v_5$};
			\node at (-0.3,0.15){$v_1$};
            \node at (2.1,0.15){$v_2$};
            \node at (1.2,0.25){$v_4$};
			\node at (0.6,0.25){$v_3$};
			\node at (2.95,-0.7){$v_7$};
			\node at (2.95,-1.4){$v_8$};
            \node at (1.8,0.95){$v_{10}$};
			\node at (0,0.95){$v_9$};
		\end{tikzpicture}
		\caption*{$R=G_3(1,s,1,q-1)$}
		\label{fig:12.3}
	\end{subfigure}
\caption{ }
	\label{fig:12}
\end{figure}
 Denote by $\rho_1=\rho(L)$ and $\rho_2=\rho(R)$. Since $T^*_{q+1}$ is a proper subgraph of $L$, it follows from Claim 1 that $$\rho_1>\rho(T^*_{q+1})=\sqrt{q+2}.$$ Let $X=(x_1,x_2,...,x_n)^T$ be the Perron vector of $L$ with $x_i$ corresponding to the vertex $v_i$, and let $Y=(y_1,y_2,...,y_n)^T$ be the Perron vector of $R$ with $y_i$ corresponding to the vertex $v_i$.

    Suppose $s\geq q$. By symmetry and $\rho_1X=A(L)X$, we have
 \begin{eqnarray}
    \rho_1(x_{2}-x_{7})&=&qx_{7}+x_{4}-x_{2}-x_{8}
    > x_{7}-x_{2}+x_{4}-x_{8},\label{eqhhh4}\\
    \rho_1(x_{3}-x_{4})&=&x_{1}+x_{4}-x_{2}-x_{3},\label{eqhhh5}\\
    \rho_1(x_{5}-x_{7})&=&x_{1}-x_{2}+x_{6}-x_{8},\label{eqhhh6}\\
    \rho_1(x_6-x_8)&=&x_5-x_7,\label{eqhhh7}\\
    \rho_1(x_{1}-x_{2})&=&sx_{5}+x_{3}-qx_{7}-x_{4}.\label{eqhhh8}
\end{eqnarray}
By (\ref{eqhhh4}), we have
\begin{align*}
    (\rho_1+1)(x_{2}-x_{7})&> x_{4}-x_{8}. 
\end{align*}
If $x_{4}\leq x_{8}$, then applying Lemma \ref{lemma13}, we have $\rho(L-v_3v_4+v_3v_8)>\rho_1$. On the other hand, applying  Lemma \ref{lemma12}, we have $\rho(L_{v_3v_4}-v_8)<\rho_1$. Since $L_{v_3v_4}-v_8\cong L-v_3v_4+v_3v_8$, we get a contradiction. Therefore, we have
\begin{align*}
 x_{4}> x_{8} \quad\text{ and } \quad  x_{2}>x_{7}.\label{equ9}
\end{align*}

Note that (\ref{eqhhh5}) implies   $$(\rho_1+1)(x_{3}-x_{4})=x_{1}-x_{2},$$ and  (\ref{eqhhh6}), (\ref{eqhhh7}) imply $$(\rho^2_1-1)(x_{5}-x_{7})=\rho_1(x_{1}-x_{2}).$$ Since $s\ge q$,  (\ref{eqhhh8}) leads to
\begin{align*}
    \rho_1(x_{ 1}-x_{ 2})
    \geq q(x_{ 5}-x_{ 7})+x_{ 3}-x_{ 4}
    =\Big(\frac{q\rho_1}{\rho_1^2-1}+\frac{1}{\rho_1+1}\Big)(x_{ 1}-x_{ 2}),
\end{align*}
which is equivalent with
 $$\Big(\rho_1-\frac{q\rho_1}{\rho_1^2-1}-\frac{1}{\rho_1+1}\Big)(x_{ 1}-x_{ 2})\geq0.$$ Since $$\rho_1-\frac{q\rho_1}{\rho_1^2-1}-\frac{1}{\rho_1+1}>0\quad\text{for}\quad \rho_1>\sqrt{q+2},$$ we have
 \begin{equation}\label{eqx1}
 x_{ 1}\geq x_{ 2}>x_{ 7}.
 \end{equation}
Since $R\cong L- v_7 v_8 +  v_1 v_8$,
 applying Lemma \ref{lemma13}, we have the right inequality in (\ref{eqc2}).

Now suppose $q=s+1$. Let $r =s-1$. Then $s=r+1$ and $q=r+2$. By Claim 1, we have $$\rho_1>\rho(T^*_{q+1})=\sqrt{r+4}\quad \text{and}\quad \rho_2>\rho(T^*_{r+2})=\sqrt{r+3}.$$
 By symmetry and $\rho_1X=A(L)X$, we have
\begin{equation*}
    \begin{cases}
        \text{$\rho_1x_{ 3}=x_{ 1}+x_{ 4}$,} &\quad\text{$\rho_1x_{ 4}=x_{ 2}+x_{ 3}$,}\\
        \text{$\rho_1x_{ 1}=(r+1)x_{ 5}+x_{ 3}$,} &\quad\text{$\rho_1x_{ 2}=(r+2)x_{ 7}+x_{ 4}$,}\\
        \text{$\rho_1x_{ 5}=x_{ 1}+x_{ 6}$,} &\quad\text{$\rho_1x_{ 7}=x_{ 2}+x_{ 8}$,}\\
        \text{$\rho_1x_{ 6}=x_{ 5}$,} &\quad\text{$\rho_1x_{ 8}=x_{ 7}$,}\\
    \end{cases}
\end{equation*}
which lead to (see Appendix A.1)
\begin{equation}\label{eqa1}
\rho_1^6-(2r+7)\rho_1^4+(r+3)(r+4)\rho_1^2-1=0.
\end{equation}  By $\rho_2Y=A(R)Y$, we have
\begin{equation*}
    \begin{cases}
        \text{$\rho_2y_{ 3}=y_{ 1}+y_{ 3},$}\\
        \text{$\rho_2y_{ 1}=(r+1)y_{ 5}+y_{ 3}+y_{ 9},$}\\
        \text{$\rho_2y_{ 5}=y_{ 1}+y_{ 6},$}\\
        \text{$\rho_2y_{ 6}=y_{ 5},$}\\
        \text{$\rho_2y_{ 9}=y_{ 1},$}\\
    \end{cases}
\end{equation*}
which lead to (see Appendix A.2)
\begin{equation}\label{eqa2}
\rho_2^4-(r+4)\rho_2^2-\rho_2+1=0.
 \end{equation}
Notice that the function  $g(x)=x^4-(r+4)x^2-x+1$ is increasing when $x>\sqrt{r+3}$. Since  $g(\sqrt{r+4})<0$, we have
\begin{equation}\label{eqc14}
\rho_2>\sqrt{r+4}.
\end{equation}
Let $$f(x)=x^6-(2r+7)x^4+(r+3)(r+4)x^2-1.$$ Then by (\ref{eqa1}) and (\ref{eqc14}) we have
\begin{align*}
    f(\rho_2)-f(\rho_1)&=\rho_2^6-(2r+7)\rho_2^4+(r+3)(r+4)\rho_2^2-1\\
    &=\rho_2^2\left[\rho_2^4-(r+4)\rho_2^2\right]-(r+3)\left[\rho_2^4-(r+4)\rho_2^2\right]-1\\
    &=\rho_2^2(\rho_2-1)-(r+3)(\rho_2-1)-1\\
    &=\rho_2^3-\rho_2^2-(r+3)\rho_2+r+2\\
    &>0.
\end{align*}
 Since $f(x)=x^6-(2r+7)x^4+(r+3)(r+4)x^2-1$ is  increasing when $x>\sqrt{r+4},$  we have the right inequality in (\ref{eqc2}).

By (\ref{eqx1}), applying Lemma \ref{lemma13} we can see that a graph in $\{G_3(0,s,0,q): s+ q=(n-4)/2\}$ attains the minimum spectral radius if and only if $\{s,q\}=\{\lceil(n-4)/4\rceil,\lfloor(n-4)/4\rfloor\}$, i.e., $G_3(0,s,0,q)\cong H(2s+2q+4)$.
 This completes the proof of Claim 2.\\

\par

\noindent\textbf{Claim 3. }\emph{Let $s,q\geq1$ be integers. If $s> q$, then $$\rho(G_3(0,s,1,q))<\rho(G_3(1,s,0,q))$$ and
 $$\rho(G_3(0,s,1,q))<\rho(G_3(0,s+1,1,q-1)).$$ If $s=q$, then $$\rho(G_3(1,s,0,q)=\rho(G_3(0,s,1,q))<\rho(G_3(1,s-1,0,q+1)).$$
}

{\it Proof.} Let $W=G_3(0,s,1,q)$ and we label its vertices  as in Figure \ref{fig:12.4}.   Let $ \rho_3=\rho(W)$ with its
	Perron vector being  $Z=(z_1,z_2,...,z_n)^T$, where $z_i$ is corresponding to the vertex $v_i$ for $i=1,\ldots,n$.

\begin{figure}[H]
	\centering
	 \begin{tikzpicture}
			\draw[fill=black](0,0)circle(0.1)
			(0.6,0)circle(0.1)
			(1.2,0)circle(0.1)
			(1.8,0)circle(0.1)
			(1.8,-0.7)circle(0.1)
			(1.8,-1.4)circle(0.1)
			(2.6,-0.7)circle(0.1)
			(2.6,-1.4)circle(0.1)
			(0,-0.7)circle(0.1)
			(0,-1.4)circle(0.1)
			(-0.8,-0.7)circle(0.1)
			(-0.8,-1.4)circle(0.1)
			(1.8,0.7)circle(0.1);
			\draw(0,0)--(1.8,0)
			(0,0)--(0,-1.4)
			(0,0)--(-0.8,-0.7)
			(-0.8,-0.7)--(-0.8,-1.4)
			(1.8,0)--(1.8,-1.4)
			(1.8,0)--(2.6,-0.7)
			(2.6,-0.7)--(2.6,-1.4)
			(1.8,0.7)--(1.8,0);		
			\draw[decorate,decoration={brace,mirror,amplitude=0.2cm}](-0.8,-1.55)--(0,-1.55);
			\draw[decorate,decoration={brace,mirror,amplitude=0.2cm}](1.8,-1.55)--(2.6,-1.55);
			\node at (-0.4,-1.95){$s$};
			\node at (2.2,-1.95){$q$};
			\node at (-0.4,-1.3){$\cdots$};
			\node at (2.2,-1.3){$\cdots$};
			\node at (-1.2,-1.4){$v_6$};
			\node at (-1.2,-0.7){$v_5$};
			\node at (0,0.3){$v_1$};
			\node at (0.6,0.3){$v_3$};
			\node at (1.2,0.3){$v_4$};
			\node at (2.1,0.2){$v_2$};
			\node at (2.95,-0.7){$v_7$};
			\node at (2.95,-1.4){$v_8$};
			\node at (1.8,0.98){$v_{9}$};
		\end{tikzpicture}
		\caption{The graph $W=G_3(0,s,1,q)$}
		\label{fig:12.4}
\end{figure}
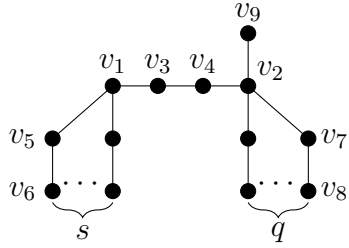

Suppose $s> q$. By $\rho_3Z=A(W)Z$,
 we have
 \begin{eqnarray}
 \rho_3(z_{3}-z_{4})&=&z_{1}+z_{4}-z_{2}-z_{3},\label{eqc31}\\
 \rho_3(z_{5}-z_{7})&=& z_{1}+z_{6}-z_{2}-z_{8},\label{eqc32}\\
 \rho_3(z_{1}-z_{2})&=&sz_{5}+z_{3}-qz_{7}-z_{4}-z_9.\label{eqc33}
  \end{eqnarray}
By (\ref{eqc31}), we have $$z_{3}-z_{4}=\frac{1}{\rho_3+1}(z_{1}-z_{2}).$$
By (\ref{eqc32}), we have
  $$\rho_3^2(z_{5}-z_{7})=\rho_3(z_{1}-z_{2})+\rho_3z_6-\rho_3z_8=\rho_3(z_{1}-z_{2})+z_{5}-z_{7},$$ which implies  $$z_{5}-z_{7}=\frac{\rho_3}{\rho_3^2-1}(z_{1}-z_{2}).$$
  Since $\rho_3z_9=z_2$ and $\rho_3z_7=z_2+z_8$, we have $z_7>z_9$.
  Now by (\ref{eqc33}) we have
    \begin{align*}
        \rho_3(z_{1}-z_{2})
         >s(z_{5}-z_{7})+z_{3}-z_{4}
         =\Big(\frac{s\rho_3}{\rho_3^2-1}+\frac{1}{\rho_3+1}\Big)(z_{1}-z_{2}),
    \end{align*}
which leads to
 $$(\rho_3-\frac{s\rho_3}{\rho_3^2-1}-\frac{1}{\rho_3+1})(z_{1}-z_{2})>0.$$
 Denote by $$h(x)=x-\frac{sx}{x^2-1}-\frac{1}{x+1}.$$
 Then $h(\sqrt{s+2})>0$ and $h(x)$ is increasing when $x>1$. Recall that $\rho_3>\rho(T^*_{s+1})=\sqrt{s+2}$. We have
 \begin{equation}\label{eqc37}
 h(\rho_3)=\rho_3-\frac{s\rho_3}{\rho_3^2-1}-\frac{1}{\rho_3+1}>0.
 \end{equation}  Hence, $z_{1}>z_{2}$.

 Applying Lemma \ref{lemma13},  we have
  $$\rho_3=\rho(G_3(0,s,1,q))<\rho(G_3(0,s,1,q)-v_2 v_9+v_1 v_9)=\rho(G_3(1,s,0,q)).$$
Similarly, considering the graph obtained from $W$ by cutting a $P_2$ attached to $v_2$ and attaching it to $v_1$, we have
$$\rho(G_3(0,s,1,q))<\rho(G_3(0,s+1,1,q-1)).$$

   Suppose $s=q$.  By $\rho_3Z=A(W)Z$ we have
   \begin{eqnarray*}
   \rho_3(z_9-z_{5})&=&z_{2}-z_{1}-z_{6},\label{eqc34}\\
   \rho_3(z_{4}-z_{3})&=&z_{2}+z_{3}-z_{1}-z_{4},\label{eqc35}\\
   \rho_3(z_{2}-z_{1})&=&sz_{7}+z_{4}+z_9-sz_{5}-z_{3}.\label{eqc36}
   \end{eqnarray*}
   Similarly as above, we have  $z_{7}>z_9$  and
    \begin{align*}
        \rho_3^2(z_9-z_{5})=\rho_3(z_{2}-z_{1}-z_{6})
         =s(z_{7}-z_{5})+z_9-z_{5}+z_{4}-z_{3}
         >(s+1)(z_9-z_{5})+z_{4}-z_{3},
    \end{align*}
    which is equivalent with
    \begin{align}
        \left[\rho_3^2-(s+1)\right](z_9-z_{5})>z_{4}-z_{3}.\label{equ10}
    \end{align}

    On the other hand, by (\ref{eqc35}) and (\ref{eqc36}) we have $$z_{2}-z_{1}=(\rho_3+1)(z_{4}-z_{3})$$
    and
    \begin{align*}
        \rho_3(z_{2}-z_{1})
         >s(z_9-z_{5})+z_{4}-z_{3},
    \end{align*}
which lead to
    \begin{align}
        \left[\rho_3(\rho_3+1)-1\right](z_{4}-z_{3})>z_9-z_{5}.\label{equ11}
    \end{align}
    By (\ref{equ10}) and (\ref{equ11}), we have $$\left[\rho_3^2-(s+1)-\frac{1}{\rho_3(\rho_3+1)-1}\right](z_9-z_{5})>0.$$
  Using similar arguments as the derivation of (\ref{eqc37}), we have $$\rho_3^2-(s+1)-\frac{1}{\rho_3(\rho_3+1)-1}>0,$$ since $\rho_3>\rho(T^*_{s+1})=\sqrt{s+2}$. Therefore, $z_9>z_{5}$. Applying Lemma \ref{lemma13}, we have $$\rho(G_3(1,s,0,q)=\rho(G_3(0,s,1,q)=\rho(W)<\rho(W-v_5v_6+v_9v_6)=\rho(G_3(1,s-1,0,q+1)).$$ This completes the proof of Claim 3.\\

  \par

 Now we are ready to present  the proof for the case $k=n-2$ of Theorem \ref{th1}. \\
 \par

{\it Proof for the case $k=n-2$ of Theorem \ref{th1}.}
  Suppose $G$ attains the minimum spectral radius  in $\mathcal{G}_{n,n-2}$ with $n\ge 10$. Applying Theorem \ref{th2}, $G$ is a tree.

  Suppose $S$ is a maximum dissociation set of $G$  such that $G[S]$ consists of $\gamma$ isolated vertices $x_1,\ldots,x_{\gamma}$ and $\tau$ disjoint edges $y_1z_1,\ldots, y_{\tau}z_{\tau}$. Let  $V(G)\backslash S=\{ v_1,v_2\}$.
   Since $G$ is connected,   the vertex $x_{i}$ is adjacent to at least one of $\{v_1,v_2\}$ for all $i\in\{1,\ldots,\gamma\}$.  Similarly, the edge $y_jz_j$  is adjacent to at least one of $\{v_1,v_2\}$ for all $j\in\{1,\ldots,\tau\}$.

     We claim  that   $G\cong G_i(r,s,p,q)$ with $i\in\{1,2,3,4\},$ where $(r,s),(p,q)\not\in\{(0,0),(1,0)\}$ in the case $i=1$, and  $(r,s)\ne (0,0)$, $(p,q)\ne (0,0)$ in the case $i=2$. In fact, if $v_1 v_2\in E(G)$, then $G\cong G_1(r,s,p,q)$. Moreover,  we have $(r,s),(p,q)\not\in\{(0,0),(1,0)\}$, since $diss(G)=n-2$.  If $v_1 v_2\not\in E(G)$, then either there is an isolated vertex in $G[S]$ adjacent to both of  $v_1$ and $v_2$, or there is an edge in $G[S]$ adjacent to both of  $v_1$ and $v_2$. In the former case, we have $G\cong G_2(r,s,p,q)$ with $(r,s)\ne (0,0)$ and $(p,q)\ne (0,0)$, since $diss(G)=n-2$; while in the latter case, we have $G\cong   G_i(r,s,p,q)$ with $i\in \{3,4\}$.

     Next we distinguish two cases to prove  $G\cong G_3(r,s,p,q)$ for some nonnegative integers $r,s,p,q$.

	{\it  Case 1.}  $G$ contains an internal path  with $v_1$ and $v_2$ being its ends. Then $$G\cong G_i(r,s,p,q)\quad \text{ with}\quad i\in \{1,2,3\}.$$

 Suppose $ G\cong G_1(r,s,p,q)$ with $(r,s),(p,q)\not\in\{(0,0),(1,0)\}$. Since $n\ge 10$,  at least one of $r+2s$, $p+2q$ is larger than or equal to 4, say, $r+2s\ge 4$.
 Let $T$ be the tree obtained from $G_{v_1v_2}^{(2)}$ by deleting two vertices from the leaves and the edges attached to $v_1$. Then  $T \cong G_3(r',s',p,q)$ for some integers $r'$ and $s'$. Moreover, we have $diss(T)=n-2$. Applying Lemma \ref{lemma3} and Lemma \ref{lemma12}, we have $\rho(T )<\rho(G)$, which contradicts the assumption on $G$.

 Similarly, if  $ G\cong G_2(r,s,p,q)$ with  $(r,s)\ne (0,0)$, $(p,q)\ne (0,0)$, we can obtain a tree $T$ from $G_{v_1v_3} $ by deleting one vertex such that
 $$diss(T)=n-2\quad \text{and}\quad \rho(T )<\rho(G),$$
 a contradiction.

  Therefore, we have $G\cong G_3(r,s,p,q)$.

 {\it Case 2.} $G$ does not contain  any internal path  with $v_1$ and $v_2$ being its ends. Then either $G\cong G_4(r,s,p,q)$ or
 \begin{equation}\label{sep927}
 G\cong G_i(r,s,p,q) ~~\text{with}~~  i\in\{1,2,3\} ~~s.t.~~  d(v_1)\le 2 ~~\text{or}~~  d(v_2)\le 2.
  \end{equation}

 Suppose $G\cong G_4(r,s,p,q)$. If   $d(v_1)\ge3$ and $ d(v_2)\ge3$, let $T'$ be the tree obtained from $G-v_4$ by subdividing the edge $v_1v_3$. Then by Lemma \ref{lemma3} and Lemma \ref{lemma12}, we have $\rho(T')<\rho(G)$. Since $diss(T')=n-2$, we get a contradiction.
 If   $d(v_1)\le2$ or  $ d(v_2)\le2$, say, $d(v_1)=2$. If $v_1$ is adjacent to a leaf, then $G\cong G_1(1,1,p,q)$ and we can deduce a contradiction as in Case 1. If $v_1$ is attached by an edge $e$, then we construct a tree $T$ from $G_{v_2v_3}$ by  deleting the leaf incident with $e$. Again, by Lemma \ref{lemma3} and Lemma \ref{lemma12}, we have $\rho(T')<\rho(G)$ and $diss(T)=n-2$, a contradiction.

 Therefore, we have (\ref{sep927}), and hence $G\cong G_3(r,s,p,q)$ for some nonnegative integers $r,s,p,q$.\\
 \par

Now if $d(v_i)\ge 3$ and $v_i$ is adjacent to two leaves $u_1,u_2$, let $G'=G-v_iu_1+u_1u_2$. Then applying Lemma \ref{lemma2}, we have $diss(G')=n-2$ and $\rho(G')<\rho(G)$, a contradiction. Therefore, we have $r\le 1$ and $p\le 1$. Moreover, applying Claim 2, we have $r+p\le 1$.

Note that $$diss(G_3(r,s,p,q))=n-2 ~~\text{for all}~~ r,s,p,q~~s.t.~~ r+p+2s+2q=n-4.$$ If $n$ is even, then $r=p=0$ and $G\cong G_3(0,s,0,q)$. Since $G_3(0,s,0,q)\cong G_3(0,q,0,s)$, we may assume $s\ge q$. By Claim 2, we have $G\cong H(n)$.
If $n$ is odd, then we have $(r,p)=(1,0)$ or $(0,1)$. By Claim 3, we also have $G\cong H(n)$.

This completes the proof of Theorem \ref{th1}.\hspace{8.5cm} $\square$

\section*{Acknowledgement}
 This work was supported by the National Natural Science Foundation of China (No. 12171323),  Guangdong Basic and Applied Basic Research
Foundation (No. 2022A1515011995) and the Science and  Technology Foundation of Shenzhen City (No. JCYJ20210324095813036).

\section*{Appendix}
\subsection*{\bf A.1 Proof of the equation (\ref{eqa1})}
Recall that
\begin{subnumcases}
{}
\text{$\rho_1x_{3}=x_{1}+x_{4}$,}\label{16a} \\
\text{$\rho_1x_{1}=(r+1)x_{5}+x_{3}$,}\label{16b} \\
\text{$\rho_1x_{5}=x_{1}+x_{6}$,}\label{16c} \\
\text{$\rho_1x_{6}=x_{5}$}\label{16d}
\end{subnumcases}
and
\begin{subnumcases}
{}
\text{$\rho_1x_{4}=x_{2}+x_{3}$,}\label{17a}\\
\text{$\rho_1x_{2}=(r+2)x_{7}+x_{4}$,}\label{17b}\\
\text{$\rho_1x_{7}=x_{2}+x_{8}$,}\label{17c}\\
\text{$\rho_1x_{8}=x_{7}$}.\label{17d}
\end{subnumcases}
By (\ref{16d}), substituting $x_6$ with ${x_{5}}/{\rho_1}$ in (\ref{16c}), we have
\begin{equation*}
\rho_1x_{5}=x_{1}+\frac{1}{\rho_1}x_{5},
\end{equation*}
which leads to
\begin{equation*}
 x_{5}=\frac{\rho_1}{\rho_1^2-1}x_{1}.\label{19}
\end{equation*}
Combining this with (\ref{16b}), we have
\begin{equation*}
\rho_1x_{1}=\frac{(r+1)\rho_1}{\rho_1^2-1}x_{1}+x_{3},
\end{equation*}
which implies
\begin{equation*}
x_{3}=\left[\rho_1-\frac{(r+1)\rho_1}{\rho^2-1}\right]x_1=\left[\rho_1-\frac{(r+1)\rho_1}{\rho^2-1}\right](\rho_1x_{3}-x_{4}),
\end{equation*}
where the last equality follows from (\ref{16a}). It follows that
  \begin{equation}
  \left[\rho_1^2-1-\frac{(r+1)\rho_1^2}{\rho_1^2-1}\right]x_{3}=\left[\rho_1-\frac{(r+1)\rho_1}{\rho_1^2-1}\right]x_{4}\label{22}
\end{equation}

Similarly, by (\ref{17a}), (\ref{17b}), (\ref{17c}), (\ref{17d}) we have \begin{equation}
\left[\rho_1^2-1-\frac{(r+2)\rho_1^2}{\rho_1^2-1}\right]x_{4}
=\left[\rho_1-\frac{(r+2)\rho_1}{\rho_1^2-1}\right]x_{3}.\label{23}
\end{equation}

Now combining (\ref{22}) with (\ref{23}), we have
\begin{align*}
&\left[\rho_1^2-1-\frac{(r+1)\rho_1^2}{\rho_1^2-1}\right]
\left[\rho_1^2-1-\frac{(r+2)\rho_1^2}{\rho^2_1-1}\right]
=\left[\rho_1-\frac{(r+1)\rho_1}{\rho_1^2-1}\right]\left[\rho_1-\frac{(r+2)\rho_1}{\rho_1^2-1}\right]\\
\end{align*}
which is equivalent with (\ref{eqa1}).

\subsection*{\bf A.2 Proof of the equation (\ref{eqa2})}
Recall that
\begin{subnumcases}
{ }
\text{$\rho_2y_{3}=y_{1}+y_{3};$}\label{11a}\\
\text{$\rho_2y_{1}=(r+1)y_{5}+y_{3}+y_{9};$}\label{11b}\\
\text{$\rho_2y_{5}=y_{1}+y_{6};$}\label{11c}\\
\text{$\rho_2y_{6}=y_{5};$}\label{11d}\\
\text{$\rho_2y_{9}=y_{1};$\label{11e}}
\end{subnumcases}
Combining (\ref{11a}) with (\ref{11e}), we have
\begin{equation*}\label{a21}
y_3=\frac{\rho_2}{\rho_2-1}y_9.
\end{equation*}
By (\ref{11c}), (\ref{11d}) and (\ref{11e}), we have
\begin{equation*}
y_6=\frac{\rho_2}{\rho_2^2-1}y_9\quad \text{and}\quad y_5= \frac{\rho_2^2}{\rho_2^2-1}y_9.\label{equ12}
\end{equation*}
Now substituting $y_1=\rho_2y_9, y_3= {\rho_2}y_9/({\rho_2-1}), y_5= {\rho_2^2}y_9/({\rho_2^2-1})$ in (\ref{11b}), we have
\begin{equation*}
\rho_2^2y_9=\frac{(r+1)\rho_2^2}{\rho_2^2-1}y_{9}+\frac{\rho_2}{\rho_2-1}y_9+ y_{9}.
\end{equation*}
Therefore, we have
\begin{equation*}
\rho_2^2=\frac{(r+1)\rho_2^2}{\rho_2^2-1}+\frac{\rho_2}{\rho_2-1}+ 1,
\end{equation*}
which is equivalent with (\ref{eqa2}).

\end{document}